\documentclass{amsart}
\usepackage{amsmath,amssymb,amsfonts,amsthm}
\usepackage{amssymb}
\usepackage{pgf}
\usepackage{todonotes}
\usepackage{epsfig}
\usepackage{mathrsfs}
\usepackage{verbatim}
\newtheorem{theorem}{Theorem}[section]
\newtheorem{lemma}[theorem]{Lemma}

\newtheorem{problem}[theorem]{Problem}
\newtheorem{corollary}[theorem]{Corollary}

\theoremstyle{definition}

\theoremstyle{remark}

\numberwithin{equation}{section}

\renewcommand{\dim}{\mathrm{dim}}

\newcommand{\R}{\mathbb{R}}
\newcommand{\N}{\mathbb{N}}

\newcommand{\X}{\mathrm{X}}

\newcommand{\B}{\mathbf{B}}

\renewcommand{\S}{\mathbf{S}}

\begin{document}

\title{Observations on quasihyperbolic geometry modeled on Banach spaces}

\author{Antti Rasila, Jarno Talponen, Xiaohui Zhang}

\address{Department of Mathematics and Systems Analysis, Aalto University, P.O. Box 11100, FI-00076 Aalto, Finland}
\email{antti.rasila@iki.fi}
\address{University of Eastern Finland, Department of Physics and Mathematics, Box 111, FI-80101 Joensuu, Finland}
\email{talponen@iki.fi}
\address{School of Science, Zhejiang Sci-Tech University, 310018 Hangzhou, China}
\email{xhzhang31415926@126.com}

\date{\today}

\begin{abstract}
In this paper, we continue our study of quasihyperbolic metric in Banach spaces. The main results of the paper present a criterion for smoothness of 
geodesics of quasihyperbolic type metrics in Banach spaces, under a Dini type condition on the weight function, which improves an earlier result of the two first authors. We also answer to a question posed by the two first authors in an earlier paper with R. Kl\'en, and present results related to the question on smoothness of quasihyperbolic balls.
\end{abstract}

\maketitle

\section{Introduction}

The quasihyperbolic metric in the $n$-dimensional Euclidean space $\R^n$ is a natural generalization of the hyperbolic metric, which was first introduced by F.W. Gehring and his students Palka \cite{GehringPalka76} and Osgood \cite{GehringOsgood79} in 1970's. The significance of this metric arises from its several useful properties. In particular, quasihyperbolic metric is generally well-behaved in quasiconformal mappings and related classes of transformations, and it naturally arises in the generalization of the Schwarz-Pick lemma for quasiconformal mappings in $\R^n$. In addition, bounds for this metric can be obtained by using the {\it distance ratio metric}, a quantity that has a simple and natural definition and is easy to compute.

Notably, unlike the conformal modulus and related approaches relying on $n$-dimensional volume integration, the quasihyperbolic metric can be defined in a wide range of metric spaces, which includes infinite dimensional Banach spaces. This observation has led to the concept of {\it (dimension) free quasiconformality}, which was developed by V\"ais\"al\"a (see \cite{Vai99}). The definition is based on the quasiconformal Schwarz-Pick lemma, where the Gr\"otzsch modulus function is replaced with an arbitrary strictly increasing function of the non-negative real numbers onto itself. This definition coincides with other definitions of  quasiconfromality in  $\R^n$. It has led to study of this class of mappings in Banach spaces, but it is also workable in a more general metric settings (see \cite{hrwz}). Besides its role in the theory of quasiconformal mappings, the quasihyperbolic metric is also related to domain classification problems, which have independent interest and applications related to certain function spaces and partial differential equations.

Because of these important applications, there has been significant interest in the quasihyperbolic metric itself, and related concepts such as quasihyperbolic geodesics and balls. However, analytic properties of such objects are usually not easy to see directly from the definition of the quasihyperbolic distance. In this context, it is generally assumed that the quasihyperbolic metric would exhibit behavior similar to the hyperbolic metric, at least in a sufficiently small scale. For example, the quasihyperbolic metric is conformal in the sense that small quasihyperbolic balls are geometrically close to the balls of the norm metric.

The first two authors have investigated these questions in a series of articles \cite{KRT2,RasilaTalponen12,RT}, the latests of which is joint work with R. Kl\'en. The purpose of this paper is to further refine and improve this line of research by presenting a more precise condition for smoothness of  geodesics  of quasihyperbolic type weighted metrics, giving additional results on tangential properties of quasihyperbolic balls and, finally, presenting some open questions related to the connection of the quasihyperbolic metric with the geometry of the underlying Banach space.

\section{Preliminaries}

In this section, we recall certain basic definitions required to formulate our main results. We refer to the monographs 
in the references for suitable background information.

\subsection{Quasihyperbolic metric}
Let $\X$ be a Banach space with $\dim \X\ge 2$, and suppose that $\Omega \subsetneq \X$ is a domain.  For $x\in \Omega$, we denote by $d(x)$ distance $d(x,\partial \Omega)$. Then the \emph{quasihyperbolic length} of a rectifiable arc $\gamma$ in $\Omega$ is defined by
\[
\ell_k(\gamma):=\int_{I} \frac{\|d \gamma\|}{d(\gamma(t))}.
\]
The \emph{quasihyperbolic distance} of two points $x,y\in \Omega$ is the number
\[
k(x,y):=k_\Omega(x,y):=\inf_\gamma \ell_k(\gamma)
\]
where the infimum is taken over all rectifiable arcs $\gamma$ joining the points $x,y\in\Omega$.

Next we recall the following definitions that are central in the geometry of Banach spaces. 

\subsection{Uniform convexity}
 The \emph{modulus of convexity} $\delta_{\X}(\epsilon),\ 0<\epsilon \leq 2,$ is defined by
\[\delta_{\X}(\epsilon):=\inf\{1-\|x+y\|/2:\ x,y\in \X,\ \|x\|=\|y\|=1,\ \|x-y\|=\epsilon\},\]
A Banach space $\X$ is called \emph{uniformly convex} if $\delta_{\X}(\epsilon)>0$ for all $\epsilon>0$, and \emph{uniformly smooth}  if 
\[
\lim_{\tau\to 0^{+}}\frac{\rho_{\X}(\tau)}{\tau}=0.
\]
Furthermore, we call a a space $\X$ uniformly convex of power type $p\in [2,\infty)$ if $\delta_{\X}(\epsilon)\geq K\epsilon^{p}$, for some $K>0$. Note that the modulus $\delta_X$ measures the convexity of the unit ball. A set $C$ is {\em strictly convex} if it is convex and
$d(sx+(1-s)y,\partial C)>0$ for all $x,y\in \partial C,\ x\neq y,$ and $0<s<1$.

\subsection{Uniform smoothness}
The \emph{modulus of smoothness} $\rho_{\X}(\tau),\ \tau>0$ is defined by 
\[\rho_{\X}(\tau):=\sup\{(\|x+y\|+\|x-y\|)/2 -1 :\ x,y\in \X,\ \|x\|=1,\ \|y\|=\tau\}.\]
Again, we call a a space $\X$ {\it uniformly smooth} of power type $p\in [1,2]$ if $\rho_{\X}(\tau)\leq K\tau^p$, for some $K>0$. 

\subsection{LUR spaces} Recall that $\X$ is {\it locally uniformly rotund} (LUR) if for all $x,x_n \in \S_\X$,  $n\in\N$,
with $\lim_{n\to\infty} \|\frac{x+x_n }{2}\| = 1$ it follows that
\[
\lim_{n\to\infty} \| x_n - x\|=0.
\]
Note that a uniformly convex space is LUR and this in turn implies that the space is strictly convex.

By following the argument of the main lemma of \cite{MartioVaisala}, we see the following
useful observations. If $\lambda$ and $\gamma$ are rectifiable paths with unit quasihyperbolic speed in a convex domain
of a strictly convex space and
\begin{equation}\label{eq:  unit}
\left\|\frac{\lambda' + \gamma'}{2d(\frac{\lambda+\gamma}{2})}\right\|=1
\end{equation}
then there is a representation
\begin{equation}\label{eq:  presentation}
\lambda' (t)= F(t) d(\lambda(t) ),\quad \gamma' (t) = F(t) d(\gamma(t))
\end{equation}
at points $t$ of differentiability of these paths. Here $F(t)$ is a norm-$1$ vector. Moreover,
\begin{multline*}
\left\|\frac{\lambda' + \gamma'}{2d(\frac{\lambda+\gamma}{2})}\right\|
\leq \left\|\frac{\lambda' + \gamma'}{2\frac{d(\lambda)+d(\gamma)}{2}}\right\|
\leq \frac{\|\lambda' \|+ \|\gamma' \|}{d(\lambda)+d(\gamma)}\\
=\left\|\frac{F (d(\lambda) + d(\gamma))}{d(\lambda)+d(\gamma)}\right\|
=\|F\|
=\left\|\frac{F+F}{2}\right\| = \left\| \frac{\frac{\lambda'}{d(\lambda)}+\frac{\gamma'}{d(\gamma)}}{2}\right\|.
\end{multline*}
The first inequality follows from the concavity of the distance function and the first equality from
\eqref{eq:  presentation}.
Suppose that $\lambda$ and $\gamma$ have unit quasihyperbolic speed. Then we have
\[\left\|\frac{\lambda'}{d(\lambda)}+\frac{\gamma'}{d(\gamma)}\right\|
\leq \frac{\|\lambda'\|}{d(\lambda)}+\frac{\|\gamma'\|}{d(\gamma)}
=\frac{d(\lambda)}{d(\lambda)}+\frac{d(\gamma)}{d(\gamma)}=2.\]
The geometric interpretation of these facts is that the quasihyperbolic length of point-wise average of paths is dominated by the average of the quasihyperbolic lengths of the
mentioned paths. If $\X$ is a strictly convex space then the above readily yields that quasihyperbolic geodesics are unique and that quasihyperbolic balls are strictly convex.

\begin{lemma}\label{lm: LUR}
Let $\X$ be a LUR Banach space. Suppose that $x,y_n \in \X$, $x\neq 0$, $n\in\N$, are vectors satisfying
\[\lim_{n\to\infty} \|x\|+ \|y_n \| - \|x+y_n \| =0\]
where $(y_n )$ is a norm bounded sequence.
Then $d(y_n ,[x])\to 0$ as $n\to\infty$.
\end{lemma}
\begin{proof}
Let $x$ and $y_n$ be vectors as above. Without loss of generality $\|x\|=1$.
According to the Hahn-Banach theorem we may fix
$f_n \in \X^*$, $\|f_n \|=1$, such that $f_n (x+ y_n )=\|x+y_n \|$ for $n\in\N$.
Then $f_n (x) \to 1$ and $f_n (y_n ) - \|y_n \| \to 0$ as $n\to\infty$.

Write $(f_n \otimes x)(z)=f_n (z)x$, $z\in \X$. Note that
\[\|(f_n \otimes x)(x+y_n )\|=|f_n (x+y_n )| \|x\| =\|x+y_n\|.\]
Observe that
\[x +y_n = f_n (x + y_n )x -(f_n (x+y_n )x -(x+ y_n ))\] 
where
\[f_n ((f_n (x+y_n )x -(x+ y_n )))=f_n (x+y_n )f_n (x) - f_n (x+y_n ) \to 0 .\]

From any subsequence $(n_k )$ we may pick a further subsequence $(n_{k_j} )$ such that
$f_{n_{k_j}}(x + y_{n_{k_j} })\to c$ as $n\to\infty$. Then, using the LUR condition with $cx$ in place of $x$,
and noting that
\[\left\|\frac{cx + (x+ y_{n_{k_j}})}{2}\right\|\to \|cx\|,\quad j\to\infty\]
yields that $x + y_{n_{k_j} } \to cx$. This suffices for the statement of the lemma, since $(n_k )$
was arbitrary.
\end{proof}

The following lemma is from \cite{AFT}.

\begin{lemma}\cite{AFT}\label{lem:inequal_series}
let $\lambda$ be a positive real number and let $\sum_{k=0}^\infty x_k$ be a convergent series with non-negative terms. Suppose that
$$\lambda x_n \geq \sum_{k=n+1}^\infty x_k, \quad (n=0, 1, 2, \cdots) $$
Then, for $0<\alpha\leq1$ we have
$$\sum_{k=0}^\infty x_k^\alpha \leq \frac{1}{(\lambda+1)^\alpha-\lambda^\alpha}\left(\sum_{k=0}^\infty x_k\right)^\alpha$$
with equality in the case $x_k=(\lambda/(\lambda+1))^k, k\geq0$.
\end{lemma}

\subsection{Radon-Nikodym Property}
A Banach space is said to have the {\it Radon-Nikod\'ym property (RNP)}, if any rectifiable and absolutely continuous path starting from the origin can be recovered by Bochner integrating its derivative.

For basic information about these concepts we refer to \cite{Diestel} and \cite{DiestelUhl77}, 
see also \cite{FA_book}. 

%{\bf Complete required preliminaries according to previous RT papers.}

\section{Main Results}

The first of our main results is the following improvement of \cite[Theorem 3.1]{RT}.

\begin{theorem}
Let $X$ be a uniformly convex Banach space whose modulus of convexity has power type. Let $\nu$ be the modulus of continuity of the weight function $w$. We assume that the function $\nu$ satisfies the condition
\begin{equation}\label{eq: dini_condition}
\limsup_{s\to0+}\frac{\int_0^s \frac{\nu(t)}{t}dt}{\nu(s)}<\infty.
\end{equation}
Then every $d_w$-geodesic $\gamma$ is $C^1$ excluding the endpoints.
\end{theorem}

Above in \eqref{eq: dini_condition} we have a strengthening of a Dini type condition. 

\begin{proof}
It is easy to see from the condition \eqref{eq: dini_condition} that there exists a constant $C$ and a positive integer $k_0>0$ such that for all $k\geq k_0$, $k\in\mathbb{N}$,
$$\int_0^{2^{-k}} \frac{\nu(t)}{t}dt<C\nu(2^{-k}).$$
Since
$$\sum_{j=k+1}^\infty\nu(2^{-j})\leq 2 \int_0^{2^{-j}}\frac{\nu(t)}{t}dt,$$
we have that for all $k\geq k_0$
$$\sum_{j=k+1}^\infty\nu(2^{-j})\leq C\nu(2^{-k}).$$
Then by Lemma \ref{lem:inequal_series}, we get
$$\sum_{j=k_0}^\infty \nu(2^{-j})^\alpha\leq \frac{1}{(C+1)^\alpha-C^\alpha}\left(\sum_{j=k_0}^\infty \nu(2^{-j})\right)^\alpha<\infty,$$
which implies that
\begin{equation}\label{eq: nu_convergence}
\sum_{j=1}^\infty \nu(2^{-j})^\alpha<\infty.
\end{equation}

Let $$\beta(h)=\frac{1}{c}\left(\frac{2}{\omega_0}\right)^{\frac{1}{p}}\nu(3h)^{\frac{1}{p}},\quad p\geq2$$
be the same function as in Rasila-Talponen \cite{RT}. Using \eqref{eq: nu_convergence} with $\alpha=\frac{1}{p}$, we see that
$$\sum_{j=1}^\infty \beta\left(\frac{h}{2^j}\right)<\infty.$$
Now the theorem follows from the argument of Rasila-Talponen \cite{RT}.
\end{proof}

It is known that some convexity properties (e.g. uniform convexity and the RNP with strict convexity) of
the underlying Banach space are transferred to the quasihyperbolic geometry in the case of a convex domain.
It was asked in \cite{KRT} whether for a locally uniformly rotund (LUR) Banach space the quasihyperbolic-metric
in a symmetric convex domain in fact induces via a Minkowski functional a norm which is LUR.
This question is next settled affirmatively in the reflexive case.
\begin{theorem}
Let $\X$ be a LUR reflexive Banach space and let $\Omega \subset \X$ be a convex domain.
Let $x_0  , y, y_n \in \Omega$, $n\in\N$, such that
\[k(x_0 , y )= k(x_0 , y_n ),\quad n\in\N, \]
\[k(x_0 , \frac{y+y_n }{2})\to k(x_0 , y),\ n\to\infty .\]
Then $y_n \to y$ in norm as $n\to\infty$. Moreover, if $\Omega$ is symmetric then
the Minkowski functional of the ball $\B_k (0 , r)$ is an equivalent LUR norm on $\X$ for any $r>0$.
\end{theorem}
\begin{proof}
The latter part of the statement follows from the first one by first observing that $\B_k (0 , r)\subset \Omega$ is a symmetric convex bounded subset including the origin as an interior point. Therefore the norm $|||\cdot |||$ induced by the Minkowski functional is equivalent to the given norm of the Banach space $\X$. Thus, in $\B_k (0 , r)$ the norm topology and the topology induced by the quasihyperbolic metric coincide, and in particular there is no need to distinguish between different modes of sequential convergence (norm vs. quasihyperbolic).
By using the fact that the function $x\mapsto d(x, \partial \Omega)$ is bounded on $\B_k (0 , r)$ from above and from below by a positive constant, we obtain that $|||y_n ||| \to 1$ if and only if $k(0, y_n )\to 1$.

Let $x_0 ,y\in \Omega$ and $\ell = k(x_0 , y)$. Fix $y_n \in \Omega$, $n\in\N$, such that
$k(x_0 , y_n)=\ell$ for all $n\in\N$ and $k(x_0 , \frac{y+y_n }{2})\to \ell$, $n\to\infty$.

Note that $\X$ is strictly convex, being LUR.
Since $\X$ is a strictly convex reflexive Banach space and $\Omega$ is its convex domain, there is,
up to a reparametrization, a unique quasihyperbolic geodesic $\lambda$ between $x_0$ and $y$, see \cite{RasilaTalponen12, RT}. Let us investigate the unique quasihyperbolic unit speed quasihyperbolic geodesics
$\lambda, \lambda_n \colon [0, \ell ] \to \Omega$ between $x_0 , y$ and $x_0 , y_n$, respectively.
We will follow closely the arguments in \cite{MartioVaisala, RasilaTalponen12, RT};
see also the Preliminaries section for the Radon-Nikodym Property (RNP) and \cite{DiestelUhl77}.

Recall that according to the RNP of $\X$ the paths $\lambda, \lambda_n$ are differentiable a.e. and can be recovered by integrating the derivatives in the Bochner sense:
\[\lambda_n (t)=\lambda_n (0)+\int_{0}^t \lambda_{n}' (s)\ ds,\quad t\in [0,\ell].\]
According to the parametrization of the paths we have that
\[ \left\|\frac{\lambda' (t)}{d(\lambda(t))} \right\| = \left\|\frac{\lambda_{n}' (t)}{d(\lambda_{n}(t))} \right\| = 1,\quad n\in\N,\ \text{for\ a.e.}\ t\in [0,\ell].\]

We wish to show that
\[\|y - y_{n}\|\ \to 0,\quad n\to \infty .\]
It suffices to show that for \emph{any} subsequence $(n_k)$ there exists a further subsequence
$(n_{k_j})$ such that the above convergence holds when passing to this subsequence
and letting $j\to\infty$ (because then clearly $\limsup_{n\to\infty}\|y-y_n \|=0$).

Since $B_k (x_0 , \ell)$ is convex we have that $k(x_0 , \frac{y+y_{n_k}}{2})\leq \ell$.
We observe that
\[1= \frac{1}{2}\left(\frac{\|\lambda' \|}{d(\lambda)}+\frac{\|\lambda_{n_k}' \|}{d(\lambda_{n_k})}\right)
\geq \frac{\frac{1}{2}(\|\lambda' \| + \|\lambda_{n_k}' \|)}{\frac{1}{2}(d(\lambda) +d(\lambda_{n_k}))}
\geq\frac{\frac{1}{2}\left\|\lambda' + \lambda_{n_k}' \right\|}{d\left(\frac{\lambda +\lambda_{n_k}}{2}\right)}\]
and
\[\ell=\int_{0}^\ell \frac{1}{2}\left(\frac{\|\lambda' \|}{d(\lambda)}+\frac{\|\lambda_{n_k}' \|}{d(\lambda_{n_k})}\right)\geq
\int_{0}^\ell \frac{\frac{1}{2}\left\|\lambda' + \lambda_{n_k}' \right\|}{d\left(\frac{\lambda +\lambda_{n_k}}{2}\right)}
\geq k\left(x_0 , \frac{y+y_{n_k}}{2}\right) \to \ell\]
as $k\to\infty$.

It follows that
\[\frac{\frac{1}{2}\left\|\lambda' + \lambda_{n_k}' \right\|}{d\left(\frac{\lambda +\lambda_{n_k}}{2}\right)} \to 1\]
in $L^1$ and in measure as $k\to\infty$. Thus there is a further subsequence $(n_{k_j})$ such that
\[\frac{\frac{1}{2}\left\|\lambda' (t) + \lambda_{n_{k_j}}' (t) \right\|}{d\left(\frac{\lambda(t) +\lambda_{n_{k_j}}(t)}{2}\right)} \to 1\]
for a.e. $t\in [0,\ell]$ as $j\to\infty$.

Using an adaptation of considerations after \eqref{eq:  presentation}
we get
\[\left\|\frac{\lambda'}{d(\lambda)} +  \frac{\lambda_{n_{k_j}}'}{d(\lambda_{n_{k_j}})}\right\|\to 2\]
for a.e. $t$ as $j\to\infty$. It follows from the LUR assumption that
\[\left\|\frac{\lambda'}{d(\lambda)} -  \frac{\lambda_{n_{k_j}}'}{d(\lambda_{n_{k_j}})}\right\|\to 0\]
for a.e. $t$ as $j\to\infty$. Lebesgue's dominated convergence theorem then yields
\begin{equation}\label{eq: LebgDCT}
\int_{0}^\ell \left\|\frac{\lambda'}{d(\lambda)} -  \frac{\lambda_{n_{k_j}}'}{d(\lambda_{n_{k_j}})}\right\|\ dm\to 0,\quad j\to\infty .
\end{equation}
By \cite[Lemma 3.3]{MartioVaisala}, the path $\lambda_{n_{k_{j}}}$ converges uniformly to the geodesic path $\lambda$ and hence,
\begin{equation}\label{eq: pathUC}
\left\|\frac{\lambda_{n_{k_{j}}}'(t)}{d(\lambda_{n_{k_{j}}}'(t))}-\frac{\lambda_{n_{k_{j}}}'(t)}{d(\lambda'(t))}\right\|\to 0, \quad j\to\infty,
\end{equation}
uniformly. Since
\[y_{n_{k_j}}=x_0 + \int_{0}^\ell \lambda_{n_{k_j}}' (t)\ dt\]
and
\[y=x_0 +  \int_{0}^\ell \lambda' (t)\ dt \]
by the Radon-Nikodym Property, we have
\begin{eqnarray*}
\|y_{n_{k_j}}-y\|&=&\left\|\int_0^\ell(\lambda_{n_{k_{j}}}'(t)-\lambda'(t))dt\right\|\\
                 &\leq&\int_0^\ell\left\|\lambda_{n_{k_{j}}}'(t)-\lambda'(t)\right\|dt\\
                 &\leq&\max\limits_{0\leq t\leq \ell}\{d(\lambda(t))\}\int_0^\ell\left\|\frac{\lambda_{n_{k_{j}}}'(t)}{d(\lambda(t))}-\frac{\lambda'(t)}{d(\lambda(t))}\right\|dt\\
                 &\leq&\max\limits_{0\leq t\leq \ell}\{d(\lambda(t))\}\left(\int_0^\ell\left\|\frac{\lambda_{n_{k_{j}}}'(t)}{d(\lambda_{n_{k_{j}}}(t))}-\frac{\lambda'(t)}{d(\lambda(t))}\right\|dt\right.\\
                 & &\left. +\int_0^\ell\left\|\frac{\lambda_{n_{k_{j}}}'(t)}{d(\lambda_{n_{k_{j}}}(t))}-\frac{\lambda_{n_{k_{j}}}'(t)}{d(\lambda(t))}\right \|dt \right)\\
                 &\to&0
\end{eqnarray*}
which follows from \eqref{eq: LebgDCT} and \eqref{eq: pathUC}.

Thus $y_{n} \to y$ in norm as $k\to \infty$.
\end{proof}

The following result has some bearing on some previous results on the existence and smoothness of geodesics
and the smoothness of the quasihyperbolic metric in convex domains, see \cite{RasilaTalponen12, RT}.

\begin{theorem}\label{thm: unique_smooth}
Let $\X$ be a Banach space. Let $\Omega \subset \X$ be a domain and $\gamma \colon [0,\ell]\to \Omega$ be a quasihyperbolic unit speed quasihyperbolic geodesic between points $x_0 , x \in \Omega$ (and of quasihyperbolic length $\ell$). Suppose that the Gateaux derivative $(k (x_0 , \cdot))' (x)$ exists and the vector derivative
$\gamma' (\ell )=\lim_{t\to \ell^- } \frac{\gamma(\ell)-\gamma(t)}{\ell- t}$ exists. Then
\[\sup_{z \in \S_\X}  (k (x_0 , \cdot))' (x) \left[\frac{z}{d(x)}\right] =  (k (x_0 , \cdot))' (x) [\gamma' (\ell )]\]
where $\|\gamma'(\ell)\|=d(x)$.
\end{theorem}
\begin{proof}
The latter part of the claim is clear from the continuity of the distance function.
Therefore the claim reduces to  checking that
\[\|(k (x_0 , \cdot))' (x)\|_{\X^*} =  (k (x_0 , \cdot))' (x) \left[\frac{\gamma' (\ell )}{d(x)}\right]\]
where $\left\|\frac{\gamma' (\ell )}{d(x)} \right\|=1$.

Clearly
\begin{equation}\label{eq: dx}
\frac{1}{d(x)}\geq \|(k (x_0 , \cdot))' (x) \|_{\X^*} \geq  (k (x_0 , \cdot))' (x)  \left[\frac{\gamma' (\ell )}{d(x)}\right]> 0.
\end{equation}
Note that according to the definition of the geodesic $k (x_0 , \gamma(t))=t$ for each $t\in [0,\ell]$
and therefore the chain rule gives us that
\[ \frac{d}{dt} k (x_0 , \gamma(t))\Big|_{t=\ell} = (k(x_0 , \cdot ))' (x) [\gamma' (\ell)]=1.\]
This means that the first two inequalities in \eqref{eq: dx} hold as equalities. This proves the claim.
\end{proof}

\begin{corollary}
Let $\X$ be a Gateaux smooth Banach space. Let $\Omega \subset \X$ be a domain and $\gamma \colon [0,\ell]\to \Omega$ be a quasihyperbolic unit speed quasihyperbolic geodesic between points $x_0 , x \in \Omega$ (and of quasihyperbolic length $\ell$). Assume that $y=\gamma (t)$ for some $0<t<1$. Suppose $k (x_0 , \cdot)$ and $k(y,\cdot)$ are Gateaux differentiable at $x$ and $\gamma$ has left derivative at $\ell$. Then $k (x_0 , \cdot)' (x)$ and $k (y , \cdot)' (x)$ coincide.

In particular, the geometric interpretation of this fact is that  the unique tangent spaces of the spheres $\S_k (x_0, \ell)$ and $\S_k (y, \ell-t)$ at $x$ coincide.
\end{corollary}
\begin{proof}
The functionals $d(x)(k (x_0 , \cdot)' (x)),\ d(x)(k (y , \cdot)' (x)) \in \S_{\X^*}$ attain their norm at
$\frac{\gamma' (\ell )}{d(x)}$ according to Theorem \ref{thm: unique_smooth} and its proof.
Thus, the Gateaux smoothness of $\X$ together with the Smulyan lemma yields that these functionals coincide.
\end{proof}

The following result is an immediate consequence of the smoothness of quasihyperbolic balls.
 
\begin{corollary}
Suppose that $\X$ is a uniformly smooth Banach space and $\Omega \subset \X$ is a convex domain. 
Then the geometric conclusion of the previous result holds: 
If $B_k (y,s) \subset B_k (x,r)$ and $z \in S_k (x,r) \cap S_k (y,s)$ then the quasihyperbolic spheres
$S_k (x,r)$ and $S_k (y,s)$ have the same tangent space at $z$.
\end{corollary}

In particular, the assumption on $\X$ is valid above if $\X$ is a Hilbert space, e.g. $\R^n$ with the usual norm.

\begin{proof}
The statement follows from the fact that under the assumptions for any $x_0 \in \Omega$ the quasihyperbolic metric $k(x_0 ,\cdot)$ is continuously Frechet differentible in $\Omega$ away from $x_0$, see Theorem 2.7 in \cite{KRT2}. 
Indeed, recall that the quasihyperbolic balls in convex domains are convex. This means that the tangent space $T$ of $B_k (x,r)$
at $z$ exists. Moreover, the tangent space is unique according to the Frechet differentiablility of $k(x_0 ,\cdot)$. 
Similar fact holds also for $B_k (y,s)$. Clearly $T$ is also the unique tangent space of 
$B_k (y,s) \subset B_k (x,r)$ at $z$.
\end{proof}

\section{Final remarks}

Professor Beata Randrianantoanina has asked in a personal communication the following general question relating the geometry of Banach spaces to the properties of quasihyperbolic manifolds modeled on these spaces. 

\begin{problem}
Which Banach space properties can be characterized by the corresponding quasihyperbolic metric induced norm, following
\cite{KRT2}? 
\end{problem}
The authors have also previously risen the question about the characterization of reflexive Banach spaces in terms of quasihyperbolic metrics. Namely, every reflexive Banach space has the property that for every convex domain the corresponding quasihyperbolic metric space is geodesic, see \cite{RasilaTalponen12}. 
\begin{problem}
Does the converse implication hold?
\end{problem}

It is useful to observe that if \eqref{eq:  unit} holds suitably asymptotically, for a pair of sequences of paths, then  the conlusion holds asymptotically as well.
This is convenient in particular in the setting of uniformly convex spaces.

Corrigendum. The following result appears in \cite{RasilaTalponen12}.
\begin{theorem}
Let $\X$ be a Banach space, $\Omega\subsetneq\X$ a domain.
Then each $j$-ball  $\B_{j}(x_{0},r),\ x_{0}\in \Omega,$ is starlike for radii $r\leq \log 2$.
\end{theorem}
Although the statement of the theorem is correct, the proof contains a blunder
with the grouping of terms. The corrected proof goes as follows.

\begin{proof}
Let $x_{0},y\in \Omega$,  $x_{0}\neq y$, such that $j(x_{0},y)\leq \log 2$. This is to say that
\[
\frac{\|x_{0}-y\|}{d(x_{0})\wedge d(y)}\leq 1.
\]
By using simple calculations involving the triangle inequality we get
\begin{eqnarray*}
j(x_{0},ty+(1-t)x_{0})&\leq &\log\left(1+\frac{t\|x_{0}-y\|}{d(x_{0})\wedge (d(y)-(1-t)\|x_{0}-y\|)}\right)\\
&\leq &\log\left(1+\frac{t\|x_{0}-y\|}{t\|x_{0}-y\|}\right)=\log 2,\quad t>0 .
\end{eqnarray*}

\end{proof}

\subsection{Acknowledgments}
This research was supported by the Academy of Finland Project \#268009. The second named author was also financially supported by the Finnish Cultural Foundation and V\"{a}is\"{a}l\"{a} Foundation. The third named author was also supported by NNSF of China Project No. 11601485.

\end{document}